\documentclass[11 pt]{article}
\usepackage[margin=2.5cm]{geometry}
\usepackage{amsmath,amsthm,amsfonts,amssymb}

\usepackage{graphicx,latexsym}
\usepackage{lscape}
\usepackage[usenames,dvipsnames]{color}
\def\keyw{\par\medskip\noindent {\it \textbf{Keywords}:}\enspace\ignorespaces}
\newtheorem{thm}{Theorem}[section]
 \newtheorem{cor}[thm]{Corollary}
 \newtheorem{lem}[thm]{Lemma}

 \newtheorem{defn}[thm]{Definition}

\numberwithin{equation}{section}

\newtheorem{discu}{Discussion:}

\newtheorem{conje}{Conjecture:}

\begin{document}
\date{}

\title{Embedding onto Wheel-like Networks}

\author{
R. Sundara\ Rajan $^{a}$
\and
T.M. Rajalaxmi $^{b}$
\and
Sudeep Stephen $^{c}$
\and
A. Arul Shantrinal $^a$
\and
K. Jagadeesh Kumar $^{a}$
}
\date{}

\maketitle
\vspace{-0.8 cm}
\begin{center}
$^a$ Department of Mathematics, Hindustan Institute of Technology and Science, Chennai, \\ India, 603 103\\
{\tt vprsundar@gmail.com} ~~~~~{\tt shandrinashan@gmail.com}  ~~~~~{\tt jagadeeshgraphs@gmail.com}\\
\medskip

$^b$ Department of Mathematics, SSN College of Engineering, Chennai, India, 603 110\\
{\tt laxmi.raji18@gmail.com}\\
\medskip

$^c$ Department of Computer Science, National University of Singapore, 13 Computing Drive, Singapore 117417\\
{\tt sudeep.stephen@nus.edu.sg}\\
\end{center}

\begin{abstract}
One of the important features of an interconnection network is its ability to efficiently simulate programs or parallel algorithms written for other architectures. Such a simulation problem can be mathematically formulated as a graph embedding problem. In this paper we compute the lower bound for dilation and congestion of embedding onto wheel-like networks. Further, we compute the exact dilation of embedding wheel-like networks into hypertrees, proving that the lower bound obtained is sharp. Again, we compute the exact congestion of embedding windmill graphs into circulant graphs, proving that the lower bound obtained is sharp. Further, we compute the exact wirelength of embedding wheels and fans into 1,2-fault hamiltonian graphs. Using this we estimate the exact
wirelength of embedding wheels and fans into circulant graphs, generalized Petersen graphs, augmented cubes, crossed cubes, M\"{o}bius cubes, twisted cubes, twisted $n$-cubes, locally twisted cubes, generalized twisted cubes, odd-dimensional cube connected cycle, hierarchical cubic networks, alternating group graphs, arrangement graphs, 3-regular planer hamiltonian graphs, star graphs, generalised matching networks, fully connected cubic networks, tori and 1-fault traceable graphs.
\end{abstract}

\keyw{Embedding, dilation, congestion, wirelength, wheel, fan, friendship graph, star, median, hamiltonian}

\section{Introduction}
Graph embedding is a powerful method in parallel computing that maps a guest network $G$ into a host network $H$ (usually an interconnection network). A graph embedding has a lot of applications, such as processor allocation, architecture simulation, VLSI chip design, data structures and data representations, networks for parallel computer systems, biological models that deal with visual stimuli, cloning and so on \cite{FaLiJi05,KuCh17,Xu13,RaRaLiSe18}.

The performance of an embedding can be evaluated by certain cost criteria, namely the dilation, the edge congestion and the wirelength. The \textit{dilation} of an embedding is defined as the maximum distance between pairs of vertices of host graph that are images of adjacent vertices of the guest graph. It is a measure for the communication time needed when simulating one network on another. The \textit{congestion} of an embedding is the maximum number of edges of the guest graph that are embedded on any single edge of the host graph. An embedding with a large congestion faces
many problems, such as long communication delay, circuit switching and the existence of different types of uncontrolled noise. The \textit{wirelength} of an embedding is the sum of the dilations in host graph of edges in guest graph {\rm \cite{Xu13,LaWi99}}.


Ring or path embedding in interconnection networks
is closely related to the hamiltonian problem [6--9]
which is one of the well known NP-complete problems in
graph theory. If an interconnection network has a hamiltonian
cycle or a hamiltonian path, ring or linear array can
be embedded in this network. Embedding of linear arrays
and rings into a faulty interconnection network is one
of the central issues in parallel processing. The problem
is modeled as finding fault-free paths and cycles
of maximum length in the graph \cite{PaLiKi07}.

The wheel-like networks plays an important role in the circuit layout and interconnection
network designs. Embedding of wheels and fans in interconnection networks is closely related to 1-fault hamiltonian problem. A graph $G$ is called $f$-fault hamiltonian if there
is a cycle which contains all the non-faulty vertices and
contains only non-faulty edges when there are $f$ or less
faulty vertices and/or edges. Similarly, a graph $G$ is called
$f$-fault traceable if for each pair of vertices $u$ and $v$, there
is a path from $u$ to $v$ which contains all the non-faulty vertices
and contains only non-faulty edges when there are
$f$ or less faulty vertices and/or edges. We note that if a
graph $G$ is hypohamiltonian, hyperhamiltonian or almost
pancyclic then it is 1-fault hamiltonian \cite{ArMaRaRa11} and it has been well studied in \cite{ChTsHsTa04,ArMaRaRa11,TsLi08}.

The rest of the paper is organized as follows: Section 2 gives definitions and other preliminaries. In Section 3, we compute the dilation, congestion and wirelength of embedding onto wheel-like networks. Finally, concluding remarks and future works are given in Section 4.

\section{Preliminaries}
In this section we give basic definitions and preliminaries related to embedding problems.

\begin{defn}{\rm \cite{BeChHaRoSc98}}
Let $G$ and $H$ be finite graphs. An \textit{embedding} of $G$ into $H$ is a pair $(f,P_f)$
defined as follows:
\vspace{-0.1 cm}
\begin{enumerate}
\item $f$ is a one-to-one map: $V(G)\rightarrow V(H)$
\item $P_f$ is a one-to-one map from $E(G)$ to $%
\{P_{f}(e):P_{f}(e)$ is a path in $H$ between $f(u)$ and $%
f(v)$ for $e=(uv)\in E(G)\}.$
\end{enumerate}
\end{defn}

By abuse of language we will also refer to an embedding $(f,P_f)$ simply by $f$. The \textit{expansion} of an embedding $f$ is the ratio of the number of vertices of $H$ to the number of vertices of $G$. In this paper, we consider embeddings with expansion one.

\begin{defn}{\rm \cite{BeChHaRoSc98}}
Let $f$ be an embedding of $G$ into $H$. If $e=(uv)\in E(G)$, then the length of $P_{f}(e)$ in $H$ is called
the \textit{dilation} of the edge $e$ denoted by $dil_f(e)$. Then
$$dil(G,H)=\underset{f:G\rightarrow H}\min ~~\underset{e\in E(G)}\max ~dil_f(e).$$
\end{defn}

\begin{defn}{\rm \cite{BeChHaRoSc98}}
Let $f$ be an embedding of $G$ into $H$. For $e\in E(H)$, let $EC_{f}(e)$ denotes the number of edges $(xy)$ of $G$
such that $e$ is in the path $P_{f}(xy)$ between $f(x)$ and $f(y)$ in $H$. In other words,
$EC_{f}(e)=\left\vert \left\{ (xy)\in E(G):e\in P_{f}(xy)\right\} \right\vert.$ Then
$$EC(G,H)=\underset{f:G\rightarrow H}\min ~~\underset{e\in E(H)}\max ~EC_{f}(e).$$
\end{defn}
%
\noindent Further, if $S$ is any subset of $E(H)$, then we define $EC_{f}(S)=\underset{e\in S}{\sum }EC_{f}(e)$.

\begin{defn}{\rm \cite{MaRaRaMe09}}\label{wirelengthdefinition}
Let $f$ be an embedding of $G$ into $H$. Then the wirelength of embedding $G$ into $H$
is given by
$$ WL(G,H)=\underset{f:G\rightarrow H}\min ~\underset{e\in E(G)}{\sum }dil_f(e)=\underset{f:G\rightarrow H}\min ~\underset{e\in
E(H)}{\sum }EC_{f}(e).$$
\end{defn}

\noindent An illustration for dilation, congestion and wirelength of an embedding torus $G$ into a path $H$ is given in Fig. \ref{fig1}. The dilation, the congestion, and the wirelength problem are different in the sense that an embedding that gives the minimum dilation need not give the minimum congestion (wirelength) and vice-versa. But, it is interesting to note that, for any embedding $g$, the dilation sum, the congestion sum and the wirelength are all equal.

Graph embeddings have been well studied for a number of networks [1,2, 4--7, 11, 13--34]. Even though there are numerous results and discussions on the wirelength problem, most of them deal with only approximate results and the estimation of lower bounds {\rm \cite{BeChHaRoSc98,ChTr98}. But the Congestion Lemma and the Partition Lemma \cite{MaRaRaMe09} have enabled the computation of exact wirelength for embeddings of various architectures \cite{MaRaRaMe09,ArRaRaMa11,FaJi07,HaFaZhYaQi10,RaArRaMa12,RaRaRa12}. In fact, the techniques deal with the congestion sum \cite{MaRaRaMe09} to compute the exact wirelength of graph embeddings. In this paper, we overcome this difficulty by taking non-regular graphs as guest graphs and use dilation-sum to find the exact wirelength.

\begin{figure}
\centering
\includegraphics[width=10 cm]{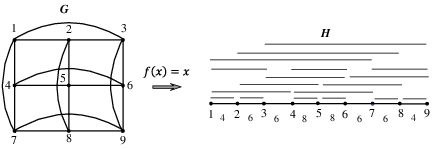}
\caption{Wiring diagram of torus $G$ into path $H$ with $dil_f(G,H)=6$, $EC_f(G,H)=8$ and $WL_f(G,H)=48$.}
\label{fig1}
\end{figure}

\begin{defn}{\rm \cite{RaQuMaWi04,My71}}
A wheel graph $W_n$ of order $n$ is a graph that
contains an outer cycle or rim of order $n-1$, and for which every
vertex in the cycle is connected to one other vertex
(which is known as the hub or center). The edges of a wheel which
include the hub are called spokes.
\end{defn}

\begin{defn}{\rm \cite{ArMaRaRa11, SaBaLiMiSi02}}
A fan graph $F_n$ of order $n$ is a graph that contains a
path of order $n-1$, and for which every vertex in the path
is connected to one other vertex (which is known as the
core). In other words, a fan graph $F_n$ is obtained from $W_n$
by deleting any one of the outer cycle edges.
\end{defn}

\begin{defn}{\rm \cite{SaBaLiMiSi02}}
A friendship graph $T_n$ of order $2n+1$ is a graph consists of $n$ triangles with exactly one common vertex called the hub or center. Alternatively, a friendship graph $T_n$ can be constructed from a wheel $W_{2n+1}$ by removing every second outer cycle edge.
\end{defn}

\begin{defn}
A windmill graph $WM_n$ of order $2n$ is obtained by deleting a vertex $v$ of degree 2 in $T_n$.
\end{defn}


\begin{defn}{\rm \cite{Xu13}}
A star graph $S_n$ is the complete bipartite graph $K_{1,n-1}$.
\end{defn}

Figures \ref{fig2}(a), \ref{fig2}(b), \ref{fig2}(c) and \ref{fig2}(d) illustrate the wheel graph $W_{12}$, fan $F_{12}$, friendship graph $T_{8}$ and windmill graph $WM_8$ respectively.

\begin{figure}
\centering
\includegraphics[width=16 cm]{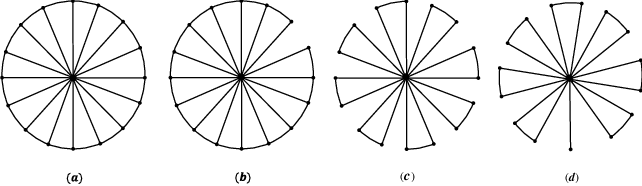}
\caption{(a)~ Wheel graph $W_{17}$ (b) Fan graph $F_{17}$ (c) Friendship graph $T_8$ and (d) Windmill graph $WM_8$}
\label{fig2}
\end{figure}


\begin{figure}
\centering
\includegraphics[width=13.0 cm]{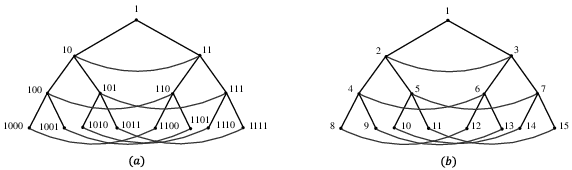}
\caption{$(a)$ $HT(4)$ with binary labels ~~$(b)$ $HT(4)$ with decimal labels}
\label{fig3}
\end{figure}

\begin{defn}{\rm \cite{GoSe81}}
The basic skeleton of a hypertree is a complete binary tree $T_r$, where $r$ is the level of a tree. Here the nodes of the tree are numbered as follows: The root node has label 1. The root is said to be at level 1. Labels of left and right children are formed by appending a 0 and 1, respectively, to the label of the parent node, see Fig. \ref{fig3}$(a)$. The decimal labels of the hypertree in Fig. \ref{fig3}$(a)$ are depicted in Fig. \ref{fig3}$(b)$. Here the children of the node $x$ are labeled as $2x$ and $2x+1$. Additional links in a hypertree are horizontal and two nodes in the same level $i$ of the tree are joined if their label difference is $2^{i-2}$. We denote an $r$ level hypertree as $HT(r)$. It has $2^{r}-1$ vertices and $3~(2^{r-1}-1)$ edges.
\end{defn}

\begin{defn}{\rm \cite{RaMaRaAr13}}
For any non-negative integer $r$, the complete binary tree of height $r-1$,
denoted by $T_{r}$, is the binary tree where each internal vertex has
exactly two children and all the leaves are at the same level. Clearly, a
complete binary tree $T_{r}$ has $r$ levels. Each level $i$, $1\leq i\leq r$,
contains $2^{i-1}$ vertices. Thus, $T_{r}$ has exactly $2^{r}-1$ vertices.
The sibling tree $ST_{r}$ is obtained from the complete binary tree $T_{r}$ by adding edges (sibling edges) between
left and right children of the same parent node.
\end{defn}

\begin{defn}
The $X$-tree $XT_{r}$ is obtained from the complete binary tree $T_{r}$ by adding the consequent vertices in each level by an edge.
\end{defn}

For illustration, the sibling tree $ST(5)$ and $X$-tree $XT_{5}$ are given in Figure \ref{fig4}.

\begin{figure}
\centering
\includegraphics[width=15.0 cm]{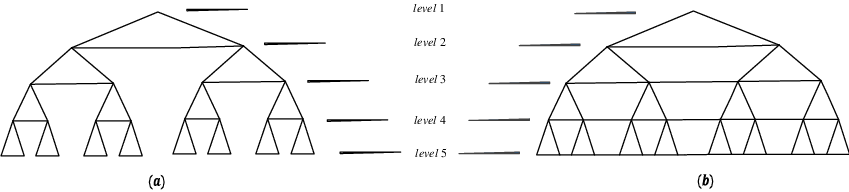}
\caption{$(a)$ Sibling tree $ST(5)$ ~~$(b)$ $X$-tree $XT_{5}$}
\label{fig4}
\end{figure}

\begin{defn}{\rm \cite{RaRaRa2012,BeCoHs95}}
The \rm{undirected circulant graph} \textit{$G(n;\pm S)$, $S\subseteq  \{1,2,\ldots,j\},$  $1\leq j\leq \lfloor n/2 \rfloor$, is a graph with the vertex
set $V=\{0,1,\ldots,n-1\}$ and the edge set $E=\{(i,k):\left\vert k-i\right\vert \equiv s (mod~n),$
$s\in S\}$.}
\end{defn}

It is clear that $G(n;\pm 1)$ is the undirected
cycle $C_{n}$ and $G(n;\pm \{1,2,\ldots,\left\lfloor n/2\right\rfloor \})$ is the complete graph $K_{n}$. The
cycle $G(n;\pm 1)\simeq C_n $ contained in $G(n;\pm \{1,2,\ldots,j\})$, $1\leq j\leq \left\lfloor n/2\right\rfloor $
is sometimes referred to as the outer cycle $C$ of $G$.

\begin{defn}{\rm \cite{RaQuMaWi04}}
Let $v$ be a vertex in $G$. The eccentricity of $v$, denoted by $\eta(v)$, is $\eta(v)=max \{d(u,v)| u \in V\}$. The maximum eccentricity is the graph diameter $d(G)$. That is, $d(G)=\max \{\eta(v):v\in V\}$. The minimum eccentricity is the graph radius $r(G)$. That is, $r(G)=\min \{\eta(v):v\in V\}$. For brevity, we denote $d(G)$ and $r(G)$ as $d$ and $r$ respectively.
\end{defn}


\paragraph{Notation:}For $u\in V(G)$, let $N_i(u)$ denotes the set of all vertices of $G$ at distance $i$ from $u$, $1\leq i\leq d$, where $d$ denotes the diameter of $G$.

\section{Main Results}
In this section we compute the dilation, congestion and wirelength of embedding onto wheel-like networks.

\subsection{Dilation}
We begin with the following definition.

\begin{defn}
A dominating set in a graph $G$ is a set of vertices $S$ such that each vertex is either in $S$ or is adjacent to a vertex in $S$. The minimum cardinality of a dominating set of $G$ is the domination number.
\end{defn}

\begin{lem}\label{lem0}
Let $G$ be a graph with domination number 1 and $H$ be a graph with $|V(G)|=|V(H)|=n$. Then $dil(G,H)\geq r$, where $r$ is the radius of $H$.
\end{lem}
\begin{proof}
Since the domination number of $G$ is 1, there exist a vertex $u\in V(G)$ such that $d(u)=n-1$. Let $f$ be an embedding from $V(G)$ to $V(H)$ and map $f(u)=v$. If eccentricity of $v$ is minimum, then $dil(G,H)\geq r$. Otherwise, $dil(G,H)\geq r+1$. Hence the proof.
\end{proof}

\begin{cor}
Let $G$ be a graph with domination number 1 and $H$ be a vertex-transitive graph with $|V(G)|=|V(H)|=n$. Then $dil(G,H)=d$, where $d$ is the diameter of $H$.
\end{cor}

We now compute the dilation of embedding wheel-like networks into hypertree and prove that the lower bound obtained in Lemma \ref{lem0} is sharp.

\begin{thm}\label{thm0}
Let $G$ be $W_n$ or $F_n$ or $T_{\frac{n-1}{2}}$ or $S_n$ and $H$ be a $l$-level hypertree $HT(l)$, where $2^l-1=n, ~l\geq 3$. Then $dil(G,H)=r=l-1$, where $r$ is the radius of $H$.
\end{thm}
\begin{proof}
Since the domination number of $G$ is 1 and by Lemma \ref{lem0}, we have $dil(G,H)\geq r$. We now prove the equality.

Label the vertices of $G$ as follows:

\begin{itemize}
  \item hub vertex as $1$;
  \item outer vertices as $2,3,\ldots,n$ consecutively start with any vertex in the clockwise or anti-clockwise direction, see Fig. \ref{fig5}(a).
\end{itemize}

\begin{figure}
\centering
\includegraphics[width=11.0 cm]{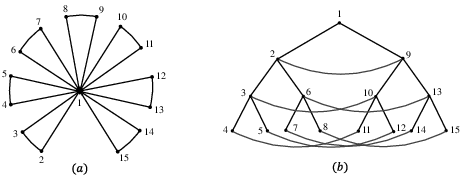}
\caption{$(a)$ Labelling of $T_7$ ~~$(b)$ Labelling of $HT(4)$}
\label{fig5}
\end{figure}

Removal of the horizontal edges in hypertree $HT(l)$ leaves a complete binary tree $T_l$. Label the vertices of $T_l$ using pre-order labeling begin with level 1 vertex, see Fig. \ref{fig5}(b). Let $f(x)=x$ for all $x\in V(G)$ and for $(ab)\in E(G)$. Let $P_f(ab)$ be a shortest path between $f(a)$ and $f(b)$ in $HT(l)$.

Since the hub vertex with label 1 in $V(G)$ is mapped into a vertex $f(1)=1$ in $V(H)$ is in level 1 gives the minimum eccentricity of $H$ and hence any edge $e=(uv)\in E(G)$ with either $u$ or $v$ as a hub vertex is mapped into a path $P_f(uv)$ in $H$ with dilation at most $l-1$, which is nothing but the radius $r$ of $H$.

We now claim that the outer edges of $G$ are mapped into a path of length at most $l-1$ in $H$. Since the graph $H$ is obtained from $T_l$, the left and right children of any parent node in level $l-1$ is connected by a path of length 2. By the labeling of pre-order traversal in $T_l$, for any parent node in level $i$, $1\leq i\leq l-2$, the right most vertex of a left node and the right node of a parent node are connected by a path length at most $l-1$ and hence the dilation of any outer edge in $G$ is at most $l-1$ in $H$. Hence the proof.
\end{proof}

Using the same analog, we prove the following result.

\begin{thm}
Let $G$ be $W_n$ or $F_n$ or $T_{\frac{n-1}{2}}$ or $S_n$ and $H$ be a $l$-level sibling tree $ST(l)$ or $l$-level $X$-tree $XT_l$, where $2^l-1=n, ~l\geq 3$. Then $dil(G,H)=r=l-1$, where $r$ is the radius of $H$.
\end{thm}

\subsection{Congestion}
In this section, we first obtain the lower bound for congestion of embedding onto wheel-like networks. Then prove that the lower bound obtained is sharp for embedding windmill graphs into circulant graphs. To prove the main result, we need the following result.

\begin{lem}\label{lem00}
Let $G$ be a graph with domination number 1 and $H$ be a graph with $|V(G)|=|V(H)|=n$. Then $EC(G,H)\geq \lceil\frac{n-1}{\vartriangle}\rceil$, where $\vartriangle$ is the maximum degree of $H$.
\end{lem}
\begin{proof}
Since the domination number of $G$ is 1, there exist a vertex $u\in V(G)$ such that $d(u)=n-1$, where $n=|V(G)|$. Let $f$ be an embedding from $V(G)$ to $V(H)$ and map $f(u)=v$. Let $S=\{e : d(v,w)=1, w\in V(H)\}$, then for any $e\in S$,
$$EC_f(e)\geq \min \left\{\frac{n-1}{\delta_0},\frac{n-1}{\delta_1},\ldots,\frac{n-1}{\delta_n}\right\}=\left \lceil\frac{n-1}{\delta_n}\right \rceil,$$
where $\delta=\delta_0 \leq \delta_1 \leq \cdots \leq \delta_n=\vartriangle$. Thus, there is at least one edge in $H$ with congestion $\left \lceil\frac{n-1}{\vartriangle}\right \rceil$. Further, for any embedding $g$ of $G$ into $H$, $EC_g(e)\geq EC_f(e)\geq \left \lceil\frac{n-1}{\vartriangle}\right \rceil$. Therefore,
$$EC(G,H)\geq \underset{g}{\min} ~EC_g(e)\geq \underset{g}{\min} ~EC_f(e)\geq \left \lceil\frac{n-1}{\vartriangle}\right \rceil.$$
Hence the proof.
\end{proof}

We now compute the edge congestion of embedding windmill graphs into circulant networks and prove that the lower bound obtained in Lemma \ref{lem00} is sharp.

\begin{thm}\label{thm00}
Let $G$ be a windmill graph $WM_{2^{n-1}}$ and $H$ be a circulant network $H(2^n;\pm \{1, 2^{n-2}\})$, where $n$ is large. Then $EC(G,H)=2^{n-2}$.
\end{thm}
\begin{proof}
Since the domination number of $G$ is 1 and by Lemma \ref{lem00}, $EC(G,H)\geq 2^{n-2}$. We now prove the equality.

Label the vertices of $G$ as follows:

\begin{itemize}
  \item hub vertex as $1$;
   \item pendent vertex as $2^n$;
  \item remaining vertices as $2,3,\ldots,2^n-1$ consecutively start with any vertex such that $(i,i+1)$ are adjacent, where $i$ even and $2\leq i \leq 2^n-2$.
\end{itemize}

Label the consecutive vertices of $H(2^n;\pm \{1\})$ in $H$ in the clockwise sense. Let $f(x)=x$ for all $x\in V(G)$ and for $(ab)\in E(G)$, let $P_f(ab)$ be a shortest path between $f(a)$ and $f(b)$ in $H$.

Since $H$ is vertex transitive, map the hub vertex $u$, which is labeled as 1 in $G$ into any vertex $v=f(u)$ in $H$. Without loss of generality, the label of $v$ as 1 $i.e., f(u)=f(1)=1=v$. Now, we map the edges in $G$ into a path $P_f$ in $H$ using the following algorithm.

\begin{itemize}
  \item For $(1i)\in E(G)$, let $P_f(1i)$ must pass through the outer cycle of $H$ in the clockwise direction, where $2\leq i\leq 2^{n-2}+1$;
   \item For $(1i)\in E(G)$, let $P_f(1i)$ must pass through the outer cycle of $H$ in the anti-clockwise direction, where $3\cdot 2^{n-2}+1\leq i\leq 2^{n}$;
  \item For $(1i)\in E(G)$, let $P_f(1i)$ must pass through an edge, which is labelled as $(1,2^{n-2}+1)$ followed by the outer cycle of $H$ in the clockwise direction, where $2^{n-2}+2\leq i\leq 2^{n-1}+1$;
      \item For $(1i)\in E(G)$, let $P_f(1i)$ must pass through an edge, which is labelled as $(1,3\cdot 2^{n-2}+1)$ followed by the outer cycle of $H$ in the anti-clockwise direction, where $2^{n-1}+2\leq i\leq 3\cdot 2^{n-2}$.
\end{itemize}

From the above algorithm, it is easy to see that the edge congestion of each edges in $H$ is at most $2^{n-2}$. At this stage, the following edges in $H$ are having $2^{n-2}$ as the edge congestion and we denote the set by $A=\{(1,2), (1,2^{n-2}+1),(2^{n-2}+1,2^{n-2}+2),(1,2^n)\}$. Now, the remaining edges $(i,i+1)$, $2\leq i\leq 2^n-2$ and $i$ is even in $E(G)$ is mapped into a path of length 1 in $H$ and it will not contribute the congestion in any of the edges in $A$. Hence the proof.
\end{proof}
\subsection{Wirelength}
We need the following lemma to prove the main result.

\begin{lem}\label{lem1}
Every 2-fault hamiltonian graph on $n$ vertices contains a hamiltonian path of length $n-1$.
\end{lem}
\begin{proof}
Let $G$ be a 2-fault hamiltonian graph. Then for $u,v\in V(G)$, $G\setminus\{u,v\}$ contains a hamiltonian cycle $C'$ of length $n-2$. Since $G$ is connected at least one of $u$ or $v$ is adjacent to a vertex $w$ in $C'$. Without loss of generality let $(u,w)\in E(G)$. Let $z$ in $V\setminus C'$ be adjacent to $w$. Now $(C'\setminus (z,w))\cup (u,w)$ is a hamiltonian path of length $n-1$ in $G$.
\end{proof}

\begin{thm}\label{thm1}
Let $G$ be a wheel and $H$ be a graph with $u$ as a median. Then $WL(G,H)\geq n-1+\delta(u)$. Equality holds if and only if $H\setminus {u}$ is hamiltonian.
\end{thm}
\begin{proof}
Let $u$ be hub of $W_n$. Map $u$ in $G$ to $u$ in $H$. Since $u$ is a median of $H$, $\delta(u)=\underset{v\in V}\sum d(u,v)=\overset{k}{\underset{i=1}\sum}|N_i(u)|$, $k\leq d$. Suppose $H\setminus {u}$ is hamiltonian. Map the outer $(n-1)$-cycle in $G$ to a hamiltonian cycle in $H\setminus {u}$. Thus
$$WL(G,H)=n-1+\overset{k}{\underset{i=1}\sum}|N_i(u)|, k\leq d.$$

Conversely, suppose $WL(G,H)= n-1+\delta(u)$. If $H\setminus {u}$ is not hamiltonian, then the cycle in $G$ cannot be mapped onto
a cycle in $H\setminus {u}$, a contradiction.
\end{proof}

Proceeding in the same way, we have the following result.
\begin{thm}\label{thm2}
Let $G$ be a fan and $H$ be a graph with $u$ as a median. Then $WL(G,H)\geq n-2+\delta(u)$. Equality holds if and only if $H\setminus {u}$ contains a hamiltonian path.
\end{thm}

\section{Concluding Remarks}
The host graphs in Theorem \ref{thm1} and Theorem \ref{thm2} cover a wide range of graphs. This has motivate us to identify interconnection networks which fall into this category:
\begin{center}
\begin{tabular}{|l|l|}
  \hline
  \hspace{3.7 cm}\textbf{Networks} & \hspace{0.3 cm}\textbf{Justification for 1-fault Tolerance} \\ \hline
  Circulant graphs $G(n;\pm S), \{1,2\}\subseteq S \subseteq \{1,2,\ldots,\lfloor\frac{n}{2}\rfloor\}$ & 1-fault hamiltonian \cite{Xu13} \\ \hline
  Generalized Petersen graphs $P(n,m)$ & hypohamiltonian$/$hyperhamiltonian \cite{Xu13,MaWaHs08} \\  \hline
  Augmented cubes $AQ_n$ & pancyclic \cite{ChSu02} \\ \hline
  Crossed cubes $CQ_n$ & almost pancyclic \cite{Ef91} \\ \hline
  M\"{o}bius cubes $MQ_n$ & $(n-2)$-fault almost pancyclic \cite{PaLiKi07,CuLa95} \\ \hline
  Twisted cubes $TQ_n$ & $(n-2)$-fault almost pancyclic \cite{PaLiKi07,FaJiLi07,FaJiLi08} \\ \hline
  Twisted $n$-cubes $T_nQ$ & 1-fault hamiltonian \cite{PaKiLi05} \\ \hline
  Locally twisted cubes $LTQ_n$ & almost pancyclic \cite{YaMeEv04} \\ \hline
  Generalized twisted cubes $GQ_n$ & $(n-2)$-fault almost pancyclic \cite{PaLiKi07} \\ \hline
  Odd dimensional cube connected cycle $CCC_n$ & 1-fault hamiltonian \cite{YaMeEv04} \\ \hline
  Hierarchical cubic networks $HCN(n)$ & almost pancyclic \cite{GhDe95} \\ \hline
  Alternating group graphs $AG_n$ & $(n-2)$-fault hamiltonian \cite{ChYaWaCh04} \\ \hline
  Arrangement graphs $A_{n,k}$ & pancyclic \cite{DaTr92} \\ \hline
  3-regular planar hamiltonian graphs & 1-fault hamiltonian \cite{WaHuHs98} \\ \hline
  $(n,k)$-star graphs $S_{n,k}$ & at most $(n-3)$-fault hamiltonian \cite{HsHsTaHs03} \\ \hline
   Generalised matching network $GMN$ & $(f+2)$-fault Hamiltonian \cite{DoYaZh09} \\ \hline
    Fully connected cubic networks $FCCN_n$ & 1-fault hamiltonian \cite{HoLi09} \\ \hline
      Tori $T(d_1,d_2,\ldots,d_n)$ & fault hamiltonian \cite{KiPa00,KiPa01} \\ \hline
  1-fault traceable graphs & 2-fault hamiltonian [Lemma \ref{lem1}] \\
  \hline
\end{tabular}
\end{center}

\section*{Acknowledgment}
The work of R. Sundara Rajan was partially supported by Project no. ECR/2016/1993, Science and Engineering Research Board (SERB), Department of Science and Technology (DST), Government of India. Further, we thank Prof. N. Parthiban, School of Computing Sciences and Engineering, SRM Institute of Science and Technology, Chennai, India for his fruitful suggestions.


\begin{thebibliography}{label}

\bibitem{FaLiJi05} {\rm J. Fan, X. Lin and X. Jia, \emph{Optimal path embedding in crossed cubes}, IEEE Transactions on Parallel and Distributed Systems, Vol. 16, no. 12, 1190-1200, 2005.}

\bibitem{KuCh17}C.-N. Kuo and Y.-H. Cheng, \emph{Cycles embedding in folded hypercubes with conditionally faulty vertices}, Discrete Applied Mathematics, Vol. 220, 55--59, 2017.

\bibitem{Xu13} {\rm J.M. Xu, \emph{Topological Structure and Analysis of Interconnection Networks}, Network theory and applications, Vol. 7, Springer Science \& Business Media, 2013.}

\bibitem{RaRaLiSe18} {\rm R.S. Rajan, T.M. Rajalaxmi, J.B. Liu and G. Sethuraman, \emph{Wirelength of embedding complete multipartite graphs into certain graphs}, Discrete Applied Mathematics, 2018. https://doi.org/10.1016/j.dam.2018.05.034}

\bibitem{LaWi99} {\rm Y.L. Lai and K. Williams, \emph{A survey of solved problems and applications on bandwidth, edgesum, and profile of graphs}, Journal of Graph Theory, Vol. 31, 75--94, 1999.}

\bibitem{LiHsJu12} {\rm T.-J. Lin, S.-Y. Hsieh and J.S.-T. Juan, \emph{Embedding cycles and paths in product networks and their applications to multiprocessor systems}, IEEE Transactions on Parallel and Distributed Systems, Vol. 23, no. 6, 1081--1089, 2012}.

\bibitem{LiLi13} {\rm M. Liu and H. Liu, \emph{Paths and cycles embedding on faulty enhanced hypercube networks}, Acta Mathematica Scientia, Vol. 33, no. 1, 227--246, 2013}.

\bibitem{ChTsHsTa04} {\rm C. Chen, C.H. Tsai, L.H. Hsu and J.M. Tan, \emph{Pn some super
fault-tolerant hamiltonian graphs}, Applied Mathematics and Computation, Vol. 148, no. 3, 729--741, 2004}.

\bibitem{GaJo79} {\rm M.R. Garey and D.S. Johnson, \emph{Computers and Intractability}, A Guide to the Theory of NP-Completeness, Freeman, San Francisco 1979.}

\bibitem{PaLiKi07} {\rm J.-H. Park, H.-S. Lim and H.-C. Kim, \emph{Panconnectivity and pancyclicity of hypercube-like interconnection networks with faulty elements}, Theoretical Computer Science, Vol. 377, no. 1-3, 170--180, 2007}.

\bibitem{ArMaRaRa11} {\rm M. Arockiaraj, P. Manuel, I. Rajasingh and B. Rajan, \emph{Wirelength of
1-fault hamiltonian graphs into wheels and fans}, Information Processing Letters, Vol. 111, 921--925, 2011}.

\bibitem{TsLi08} {\rm C.-H. Tsai and T.-K. Li, \emph{Two construction schemes for cubic hamiltonian 1-node-hamiltonian graphs}, Mathematical and Computer Modelling, Vol. 48, no. 3-4, 656--661, 2008}.

\bibitem{BeChHaRoSc98} {\rm S.L. Bezrukov, J.D. Chavez, L.H. Harper, M. R\"{o}ttger and U.P. Schroeder, \emph{Embedding of hypercubes into grids}, MFCS, 693--701, 1998.}

\bibitem{MaRaRaMe09} {\rm P. Manuel, I. Rajasingh, B. Rajan and H. Mercy, \emph{Exact wirelength of hypercube on a grid}, Discrete Applied Mathematics, Vol. 157, no. 7, 1486--1495, 2009.}

\bibitem{XuMa09} {\rm J.M. Xu and M. Ma, \emph{Survey on path and cycle embedding in some networks}, Frontiers of Mathematics in China, Vol. 4, 217--252, 2009}.

\bibitem{BeChHaRoSh00} {\rm S.L. Bezrukov, J.D. Chavez, L.H. Harper, M. R\"{o}ttger and U.P. Schroeder, \emph{The congestion of $n$-cube layout on a rectangular grid}, Discrete Mathematics, Vol. 213, 13--19, 2000.}

\bibitem{OpSo00} J. Opatrny and D. Sotteau, \emph{Embeddings of complete binary trees into grids and extended grids with total vertex-congestion 1}, Discrete Applied Mathematics, Vol. 98, 237--254, 2000.

\bibitem{ChTr98} {\rm J.D. Chavez and R. Trapp, \emph{The cyclic cutwidth of trees}, Discrete Applied Mathematics, Vol. 87, 25--32, 1998.}

\bibitem{RaQuMaWi04} {\rm I. Rajasingh, J. Quadras, P. Manuel and A. William, \emph{Embedding of cycles and wheels into arbitrary trees}, Networks, Vol. 44, 173--178, 2004.}

\bibitem{WoMa97} {\rm W.K. Chen and M.F.M. Stallmann, \emph{On embedding binary trees into hypercubes}, Journal on Parallel and Distributed Computing, Vol. 24, 132 - 138, 1995.}

\bibitem{ArRaRaMa11} {\rm P. Manuel, M. Arockiaraj, I. Rajasingh and B. Rajan, \emph{Embedding hypercubes into cylinders, snakes and caterpillars for minimizing wirelength}, Discrete Applied Mathematics, Vol. 159, no. 17, 2109--2116, 2011.}

\bibitem{RaRaRa2012} {\rm I. Rajasingh, B. Rajan and R.S. Rajan, \emph{Embedding of special classes of circulant networks, hypercubes and generalized Petersen graphs}, International Journal of Computer Mathematics, Vol. 89, no. 15, 1970--1978, 2012}.

\bibitem{FaJi07} {\rm J. Fan and X. Jia, \emph{Embedding meshes into crossed cubes}, Information Sciences,
      Vol. 177, no. 15, 3151--3160, 2007.}

\bibitem{HaFaZhYaQi10} {\rm Y. Han, J. Fan, S. Zhang, J. Yang and P. Qian, \emph{Embedding meshes into locally twisted cubes}, Information Sciences, Vol. 180, no. 19, 3794--3805, 2010.}

\bibitem{YaDoTa10} {\rm X. Yang, Q. Dong and Y.Y. Tan, \emph{Embedding meshes/tori in faulty crossed cubes}, Information Processing Letters, Vol. 110, no. 14-15, 559--564, 2010.}

\bibitem{CaKo01} {\rm R. Caha and V. Koubek, \emph{Optimal embeddings of  ladders into hypercubes}, Discrete Mathematics, Vol. 233, 65--83, 2001.}

\bibitem{Ch88} {\rm B. Chen, \emph{On embedding rectangular grids in hypercubes}, IEEE Transactions on Computers, Vol. 37, no. 10, 1285--1288, 1988.}

\bibitem{El91} {\rm J.A. Ellis, \emph{Embedding rectangular grids into square grids}, IEEE Transactions on Computers, Vol. 40, no. 1, 46--52, 1991.}

\bibitem{RoSc01} {\rm M. Rottger and U.P. Schroeder, \emph{Efficient embeddings of grids into grids}, Discrete Applied Mathematics, Vol. 108, no. 1-2, 143--173, 2001.}

\bibitem{Ts08} {\rm C.-H. Tsai, \emph{Embedding of meshes in M\"{o}bius cubes}, Theoretical Computer Science, Vol. 401, no. 1-3, 181--190, 2008.}

\bibitem{LaTs10} {\rm P.-L. Lai and C.-H. Tsai, \emph{Embedding of tori and grids into twisted cubes}, Theoretical Computer Science, Vol. 411, no. 40-42, 3763--3773, 2010.}

\bibitem{RaArRaMa12} {\rm I. Rajasingh, M. Arockiaraj, B. Rajan and P. Manuel, \emph{Minimum wirelength of hypercubes into n-dimensional grid networks}, Information Processing Letters, Vol. 112, 583--586, 2012.}

\bibitem{RaRaRa12} {\rm I. Rajasingh, B. Rajan and R.S. Rajan, \emph{Embedding of hypercubes into necklace, windmill and snake graphs}, Information Processing Letters, Vol. 112, 509--515, 2012.}

\bibitem{RaMaRaAr13} {\rm I. Rajasingh, P. Manuel, B. Rajan and M. Arockiaraj, \emph{Wirelength of hypercubes into certain trees}, Discrete Applied Mathematics, Vol. 160, 2778 - 2786, 2012.}

\bibitem{My71} {\rm B.R. Myers, \emph{Number of spanning trees in a wheel}, IEEE Transactions on Circuit Theory, Vol. 18, 280--282, 1971.}

\bibitem{SaBaLiMiSi02}Slamin, M. Ba$\check{c}$a, Y. Lin, M. Miller and R. Simanjuntak, \emph{Edge-magic total labelings
of wheels, fans and friendship graphs}, Bulletin of the Institute of Combinatorics and its Applications, Vol. 35, 89--98, 2002.


\bibitem{GoSe81} J.R. Goodman and C.H. Sequin, \emph{A multiprocessor interconnection topology}, IEEE Transactions on Computers, Vol. c-30, no. 12, 923--933, 1981.

\bibitem{BeCoHs95} {\rm J.C. Bermond, F. Comellas and D.F. Hsu, \emph{Distributed loop computer networks}, A survey: Journal of Parallel and Distributed Computing, Vol. 24, no. 1, 2 - 10, 1995.}

\bibitem{MaWaHs08}T.-C. Mai, J.-J. Wang and L.-H. Hsu, \emph{Hyperhamiltonian generalized Petersen
graphs}, Computer and Mathematics with Applications, Vol. 55, no. 9, 2076--2085, 2008.

\bibitem{ChSu02}S.A. Choudum and V. Sunitha, \emph{Augmented cubes}, Networks, Vol. 40, no. 2, 71--84, 2002.

\bibitem{Ef91}E. Efe, \emph{A variation on the hypercube with lower diameter}, IEEE Transactions on Computers, Vol. 40, no. 11, 1312--1316, 1991.

\bibitem{CuLa95}P. Cull and S.M. Larson, \emph{The M\"{o}bius cubes}, IEEE Transactions on Computers, Vol. 44, no. 5, 647--659,  1995.

\bibitem{FaJiLi07}J. Fan, X. Jia and X. Lin, \emph{Optimal embeddings of paths with various lengths in twisted cubes}, IEEE Transactions on Parallel and Distributed Systems, Vol. 18, no. 4, 511--521, 2007.

\bibitem{FaJiLi08}J. Fan, X. Jia and X. Lin, \emph{Embedding of cycles in twisted cubes with edge pancyclic}, Algorithmica, Vol. 51, no. 3, 264--282, 2008.

\bibitem{PaKiLi05}J.-H. Park, H.-C. Kim and H.-S. Lim, \emph{Fault-hamiltonicity of hypercube-like interconnection networks}, in: Proc. of IEEE International Parallel and Distributed Processing Symposium, IPDPS 2005, Denver, 2005.

\bibitem{YaMeEv04}X. Yang, G.M. Megson and D.J. Evans, \emph{Locally twisted cubes are 4-pancyclic}, Applied Mathematics Letters, Vol. 17, 919--925, 2004.

\bibitem{GhDe95}K. Ghose and K.R. Desai, \emph{Hierarchical cubic networks}, IEEE Transactions on Parallel and Distribted Systems, Vol. 6, no. 4, 427--435, 1995.

\bibitem{ChYaWaCh04}J.-M. Chang, J.-S. Yang, Y.-L. Wang and Y. Cheng, \emph{Panconnectivity, fault tolerant hamiltonicity and hamiltonian-connectivity in alternating group graphs}, Networks, Vol. 44, 302--310, 2004.

\bibitem{DaTr92}K. Day and A. Tripathi, \emph{Arrangement graphs: A class of generalized star graphs}, Information Processing Letters, Vol. 42, no. 5, 235--241, 1992.

\bibitem{WaHuHs98}J.-J. Wang, C.-N. Hung and L.-H. Hsu, \emph{Optimal 1-hamiltonian graphs}, Information Processing Letters, Vol. 65, no. 3, 157--161, 1998.

\bibitem{HsHsTaHs03}H.-C. Hsu, Y.-L. Hsieh, J.J.M. Tan and L.-H. Hsu, \emph{Fault hamiltonicity and fault hamiltonian connectivity of the $(n,k)$-star graphs}, Networks, Vol. 42, no. 4, 189--201, 2003.

\bibitem{DoYaZh09}Q. Donga, X. Yanga and J. Zhaob, \emph{Fault hamiltonicity and fault hamiltonian-connectivity of generalised matching networks}, International Journal of Parallel, Emergent and Distributed Systems, Vol. 24, no. 5, 455--461, 2009.

\bibitem{HoLi09}T.-Y. Ho and C.-K. Lin, \emph{Fault-tolerant hamiltonian connectivity and fault-tolerant hamiltonicity of the fully connected cubic networks}, Journal of Information Science and Engineering, Vol. 25, 1855--1862, 2009.

\bibitem{KiPa00} {\rm H.-C. Kim and J.-H. Park, \emph{Fault hamiltonicity of two-dimensional torus networks}, in: Proc. Japan-Korea Joint Workshop on Algorithms and Computation, 110--117, 2000}.

\bibitem{KiPa01} {\rm H.-C. Kim and J.-H. Park, \emph{Paths and cycles in d-dimensional tori with faults}, in: Workshop on Algorithms and Computation WAAC'01, Pusan, Korea, 67--74, 2001}.

\end{thebibliography}
\end{document}